\documentclass[10pt,leqno]{amsart}
\usepackage{amscd}
\usepackage{amssymb}
\usepackage{amsfonts}
\usepackage{latexsym}
\usepackage[T1]{fontenc}
\usepackage[latin1]{inputenc}
\usepackage{verbatim}
\usepackage{amsmath}
\usepackage{amsthm}
\usepackage{mathrsfs}

\newcommand{\wn}{\widetilde{\nabla}}
\theoremstyle{plain}
\newtheorem{theorem}{Theorem}[section]
\newtheorem*{theorem*}{Theorem}
\newtheorem{definition}[theorem]{Definition}
\newtheorem{lemma}[theorem]{Lemma}
\newtheorem{prop}[theorem]{Proposition}
\newtheorem{cor}[theorem]{Corollary}
\newtheorem{rem}[theorem]{Remark}
\newtheorem{ex}[theorem]{Example}
\DeclareMathOperator{\di}{d}
\begin{document}
\title{Gray identities, canonical connection and integrability}
\author{Antonio J. Di Scala and Luigi Vezzoni}
\date{\today}
\address{Dipartimento di Matematica,
Politecnico di Torino, Corso Duca degli Abruzzi 24, 10129 Torino,
Italy} \email{antonio.discala@polito.it}
\address{Dipartimento di Matematica \\ Universit\`a di Torino\\
Via Carlo Alberto 10 \\
10123 Torino\\ Italy} \email{luigi.vezzoni@unito.it}
\thanks{This work was supported by the Project M.I.U.R. ``Riemann Metrics and  Differenziable Manifolds''
and by G.N.S.A.G.A. of I.N.d.A.M.}
\begin{abstract}
We characterize quasi K\"ahler manifolds whose curvature
tensor associated to the canonical Hermitian connection satisfies the first Bianchi identity.
This condition is related with the third Gray identity and in the almost K\"ahler case implies
the integrability. Our main tool is the existence of generalized holomorphic
frames introduced by the second author previously. By using such
frames we also give a simpler and shorter proof of a Theorem of
Goldberg. Furthermore we study almost Hermitian structures
having the curvature tensor associated to the canonical Hermitian connection equal to zero.
We show some explicit examples of quasi K\"ahler structures on the Iwasawa manifold
having  the Hermitian curvature vanishing
and the Riemann curvature tensor satisfying the second Gray identity.
\end{abstract}
\maketitle
\newcommand\C{{\mathbb C}}
\newcommand\f{\mathcal{F}}
\newcommand\g{{\frak{g}}}
\renewcommand\k{{\kappa}}
\renewcommand\l{{\lambda}}
\newcommand\m{{\mu}}
\renewcommand\O{{\Omega}}
\renewcommand\t{{\theta}}
\newcommand\ebar{{\bar{\varepsilon}}}
\newcommand\R{{\mathbb R}}
\newcommand\Z{{\mathbb Z}}
\newcommand\T{{\mathbb T}}
\newcommand{\de}[2]{\frac{\partial #1}{\partial #2}}
\newcommand\w{\wedge}
\newcommand{\ov}[1]{\overline{ #1}}
\newcommand{\Tk}{\mathcal{T}_{\omega}}
\newcommand{\ovp}{\overline{\partial}}
\section{Introduction}

\noindent Quasi K\"ahler and almost K\"ahler manifolds are special classes of almost Hermitian manifolds
and can be considered as natural generalizations
of K\"ahler manifolds to the context of almost symplectic and symplectic manifolds.
It is well known that if $(M,\omega)$ is a (almost) symplectic manifold, then
there always exists an almost complex structure $J$ compatible with $\omega$.
Furthermore the choice of such an almost complex
structure is unique up to homotopy. Hence quasi K\"ahler and almost K\"ahler
structures can be consider as a tool to study (almost) symplectic
manifolds.

The interplay between the integrability of almost Hermitian structures and
the curvature has been largely studied in
the last years (see e.g. \cite{AD}, \cite{Kiirchberg} and the references therein). One
of the most important results in this topics is due to Goldberg.
Indeed Goldberg in \cite{Goldberg} proved that if the Riemann
curvature tensor of an almost K\"ahler metric $g$ satisfies the
first Gray condition, i.e. if it commutes with the almost complex structure, then $g$  is a K\"ahler metric.
Gray's conditions were introduced in \cite{Gray} and consist of some
formulae involving the curvature tensor of an almost Hermitian
metric and the associated almost complex structure.
The Goldberg theorem has been further generalized to the following formula:
\begin{equation}\label{s-s*}
s_*-s=\|\nabla\omega\|^2\,,
\end{equation}
where $s$ and $s_*$ are the scalar curvature and the $*$-scalar curvature associated to an almost
K\"ahler structure $(g,J,\omega)$ respectively
(see e.g. \cite{AD}). The classical proof of this result is based on the Weitzenb\"ock decomposition.

The research of this paper moves from \cite{donaldson,weinkove,yau}. In \cite{donaldson} Donaldson stated the following

\bigskip

\noindent \textbf{Conjecture.} \emph{Let $(M,J)$ be a compact almost complex manifold and let $\Omega$ be a taming symplectic form. Let $\sigma$ be
a smooth volume form on $M$ with $\int_{M}\sigma=\int_M \Omega^2$. Then if $\omega$ is an almost K\"ahler form with $[\omega]=[\Omega]$
and solving the Calabi-Yau equation
$$
\omega^2=\sigma\,,
$$
there are $C^\infty$ a priori bounds on $\omega$ depending only on $\Omega$, $J$ and $\sigma$.}

\bigskip
Donaldson's conjecture can be viewed as the symplectic counterpart of the Calabi conjecture proved by Yau in \cite{Calabi}.
In \cite{weinkove} Weinkove showed that a solution of the Calabi-Yau equation exists if the
Nijenhuis tensor is small in a certain sense.
In \cite{yau} Tosatti, Weinkove and Yau  studied the analogue of the Donaldson's conjecture in any dimension.
The main tool of their work is to use the canonical connection associated to an almost Hermitian structure
instead of the Levi-Civita connection (for the definition of the \emph{canonical connection} see section 2).
They proved the following

\medskip
\noindent \textbf{Theorem.} (\cite{yau}, Theorem $2$).\emph{ Let
$(M,\omega)$ be a symplectic manifold and let $J$ be an almost complex structure tamed by $\omega$.
Denote by $g$ be the almost Hermitian metric induced by $(\omega,J)$, by
$\widetilde{R}$ the curvature tensor associated
to the canonical connection of $(g,J)$, by $N$ the Nijenhuis tensor of
$J$ and by $\mathcal{R}(g,J)$ the tensor
\begin{equation}\label{tosatti}
\mathcal{R}_{i\overline{j}k\overline{l}}(g,J):=\widetilde{R}^j_{ik\overline{l}}+4N^{r}_{\ov{l}\ov{j}}\ov{N^i_{\ov{r}\ov{k}}}\,.
\end{equation}
If $\mathcal{R}(g,J)\geq 0$, then the Donaldson's conjecture holds.}

\bigskip

\noindent The last theorem imposes to study the curvature tensor associated to the canonical
connection of an almost Hermitian structure. In  this paper we will refer to this curvature tensor
as to the \emph{Hermitian curvature tensor.}

In \cite{formality} de Bartolomeis and Tomassini proved that a quasi
K\"ahler manifold always admits a special complex frame. This
result has been improved by the second author in \cite{LV1} introducing
generalized normal holomorphic frames. Such frames have been
further taken into account in \cite{LV2} to prove that if the
holomorphic bisectional curvature associated to an almost K\"ahler
metric $g$ and the holomorphic bisectional curvature associated to
the relative canonical connection coincides, then $g$ is a K\"ahler
metric. This result is not trivial, since the Hermitian curvature tensor
does not necessary satisfy the first Bianchi identity.

As first result of this paper we give a new proof of formula \eqref{s-s*}. Our proof is elementary and does not make use of
the Weitzenb\"ock decomposition, but of the existence of generalized normal holomorphic frames only.
Sections \ref{secR1}, \ref{secR2} are dedicated to the study of the Hermitian curvature tensor in quasi K\"ahler and almost
K\"ahler manifolds.
We show that in the quasi K\"ahler case this curvature tensor
satisfies the first Bianchi identity if and only
if the curvature of $g$ satisfies both the third Gray
condition and another special identity involving the derivative of
the Nijenhuis tensor. Namely,
\begin{theorem}\label{main} Let $(M,g,J,\omega)$ be a quasi K\"ahler manifold.
The Hermitian curvature tensor $\widetilde{R}$ satisfies the first
Bianchi identity
\begin{equation}
\label{bianchi}
\underset{X,Y,Z}{\mathfrak{S}}\widetilde{R}(X,Y,Z,\cdot)=0\,,\quad \mbox{for every }X,Y,Z\in\Gamma(TM)
\end{equation}
if and only if the following conditions hold:
\begin{enumerate}
\item[$1.$] the curvature tensor $R$ associated to $g$ satisfies the third Gray identity
$$
R(\overline{Z}_1,Z_2,Z_3,Z_4)=0\,,\quad\mbox{ for every } Z_1,Z_2,Z_2,Z_3\in\Gamma(T^{1,0}M)\,;
$$
\item[$2.$] we have
$$
R(Z_1,Z_2,\overline{Z}_3,\overline{Z}_4)=\frac14 F(\ov{Z}_3,Z_1,Z_2,\ov{Z}_4)$$
for
every $Z_1,Z_2,Z_3,Z_4\in\Gamma(T^{1,0}M)$,
where $F$ is the tensor
$$
F(X,Y,Z,W):=g((\nabla_{X}N)(Y,Z),W)\,,
$$
$\nabla$ is the Levi-Civita connection of $g$ and $N$ denotes the Nijenhuis tensor.
\end{enumerate}
\end{theorem}
\noindent The previous theorem allows to prove the following
\begin{cor} \label{4Apost}
Let $(M,g,J,\omega)$ be an almost K\"ahler
manifold. Assume that the Hermitian curvature tensor associated to
$(g,J)$ satisfies the first Bianchi identity \eqref{bianchi}, then
$(M,g,J,\omega)$ is a K\"ahler manifold.
\end{cor}

In section \ref{secR2} we study almost Hermitian
manifolds whose Hermitian curvature tensor vanishes. By corollary \ref{4Apost} this condition forces
a $4$-dimensional quasi K\"ahler structure
to be K\"ahler. In higher dimensions things work differently even
in the compact case. In fact we show that it is possible to construct examples of
strictly quasi K\"ahler nilmanifolds having the Hermitian curvature
equal to zero. It is important to observe that the curvature of our
examples satisfies the second Gray identity and that the tensor $\mathcal{R}(g,J)$ introduced by Tosatti, Weinkove and Yau vanishes.

\bigskip
\noindent {\sc Acknowledgments:} The authors would like to thank Simon Salamon for useful conversations.
They are also grateful to Sergio Console and Valentino Tosatti for useful suggestions and remarks.\\
\newline
\noindent{\sc Notation.} Given a differential manifold $M$, $TM$ denotes its tangent bundle. If a vector bundle $F$ is fixed, then
$\Gamma(F)$ denotes the vector space of the relative smooth sections. If $Z_i$ is a complex vector field on
a manifold $M$, then we usually write $Z_{\overline{i}}$ instead of $\overline{Z}_i$. The cyclic sum is denoted with the symbol $\mathfrak{S}$.

\section{Review}
\subsection{Almost Hermitian manifolds}
\noindent Let $M$ be a $2n$-dimensional manifold. An \emph{almost complex structure} on $M$ is an endomorphism $J$ of $TM$ satisfying $J^2=-{\rm Id}$.
An almost complex structure $J$ is said to be \emph{integrable} if the Nijenhuis tensor
$$
N(X,Y):=[JX,JY]-J[JX,Y]-J[X,JY]-[X,Y]\,,\mbox{ for }X,Y\in
\Gamma(TM)
$$
vanishes everywhere. In view of the celebrated Newlander-Nirenberg Theorem (see \cite{NN}),
$J$ is integrable if and only if it is induced by a system of holomorphic coordinates.
Any almost complex structure on $M$ induces a natural splitting of the complexified tangent bundle into
$$
TM\otimes\C=T^{1,0}M\oplus T^{0,1}M\,\,,
$$
where $T^{1,0}M$ and $T^{0,1}M$ are the eigenspaces relatively to $i$ and $-i$, respectively.
Consequently the vector bundle  $\wedge^p M\otimes\C$ of complex $p$-forms on $M$ splits as
\begin{equation*}
\label{p-forms}
\wedge^p M\otimes\C=\bigoplus_{r+s=p}\wedge^{r,s}M\,.
\end{equation*}
Since
$$
{\rm d}(\Gamma(\wedge^{r,s}M))\subseteq \Gamma(\wedge^{r+2,s-1}M\oplus\wedge^{r+1,s}M\oplus\wedge^{r,s+1}M\oplus\wedge^{r-1,s+2}M)\,,
$$
then the exterior derivative splits as
$$
{\rm d}=A+\partial+\overline{\partial}+\overline{A}\,.
$$
It is well known that $J$ is integrable if and only if $A=0$.
Furthermore, it can be useful to observe that the Nijenhuis tensor satisfies
\begin{equation}\label{proprietadiN}
N(Z_1,Z_2)\in\Gamma(T^{0,1}M)\,,\quad N(Z_1,\overline{Z}_2)=0
\end{equation}
for every $Z_1,Z_2\in\Gamma(T^{1,0}M)$. A Riemannian metric $g$ on $(M,J)$ is said to be $J$-\emph{Hermitian} if it is preserved by $J$.
In this case the pair $(g,J)$ is called
an \emph{almost Hermitian structure}. Any almost Hermitian structure $(g,J)$ induces a natural almost symplectic structure
$\omega(\cdot,\cdot):=g(J\cdot,\cdot)$.
\begin{definition}
The triple $(g,J,\omega)$ is called:
\begin{itemize}
\item a \emph{quasi K\"ahler} structure if $\ovp\omega=(\operatorname{d}\!\omega)^{1,2}=0$;
\vspace{0.1 cm}
\item an \emph{almost K\"ahler} structure if $\operatorname{d}\!\omega=0$.
\end{itemize}
\end{definition}
On the other hand, if $\omega$ is a non-degenerate $2$-form on an almost complex manifold $(M,J)$, the  we say that
$J$ is \emph{tamed} by $\omega$ if
$$
\omega(X,JX)> 0\,,\quad \mbox{for all }X\neq 0\,.
$$
In this case we can define a Riemannian metric $g$ by
$$
g(X,Y):=\frac12 (\omega(X,JY)+\omega(Y,JX))\,.
$$

The following lemma will be useful in the sequel
\begin{lemma}[\cite{LV1}, Corollary 1]\label{LVcor}
Let $(M,g,J,\omega)$ be an almost Hermitian manifold and let $\nabla$
be the Levi-Civita connection associated to $g$. Then the following hold:
\begin{enumerate}
\item[a.] If $\omega$ is a quasi K\"ahler form, then
\begin{equation}
\label{fund1}\nabla_{\overline{Z}_1}Z_2\in\Gamma(T^{1,0}M)\,,\quad \mbox{for all } Z_1,Z_2\in\Gamma(T^{1,0}M)\,;
\end{equation}
\item[b.] If $\omega$ is an almost K\"ahler form, then the Nijenhuis tensor satisfies
\begin{equation}
\label{fund2}
g(\nabla_{Z_1}Z_2,Z_3)=\frac{1}{4}g(N(Z_2,Z_3),Z_1))\,, \quad \mbox{for all } Z_1,Z_2,Z_3\in\Gamma(T^{1,0}M)\,.
\end{equation}
\end{enumerate}
\end{lemma}
\begin{proof}
It is well known that for an almost Hermitian structure $(g,J,\omega)$ the following fundamental relation holds
\begin{equation}\label{Fondamentale}
2g((\nabla_XJ)Y,Z)={\rm d}\omega(X,JY,JZ)-{\rm
d}\omega(X,Y,Z)+g(N(Y,Z),JX)\,,
\end{equation}
for every $X,Y,Z\in\Gamma(TM)$.
Formulae \eqref{fund1} and \eqref{fund2} can be obtained just by considering the complex
extension of \eqref{Fondamentale}.
\end{proof}
\subsection{The canonical connection} A linear connection on an almost Hermitian manifold $(M,g,J)$ is called
\emph{Hermitian} if it preserves  $g$ and $J$. Any
almost Hermitian manifold admits a canonical Hermitian connection $\widetilde{\nabla}$,
which is characterized by the following properties
$$
\widetilde{\nabla} g=0\,,\quad \widetilde{\nabla}J=0\,,\quad {\rm Tor}(\wn)^{1,1}=0\,,
$$
where ${\rm Tor}(\wn)^{1,1}$ denotes the $(1,1)$-part of the torsion of $\widetilde{\nabla}$. In the special case of a quasi
K\"ahler structure, $\widetilde{\nabla}$ is given by
$$
\widetilde{\nabla}=\nabla-\frac12 J\nabla J\,,
$$
where $\nabla$ is the Levi-Civita connection of $g$ (see for instance \cite{gau}). We will call $\wn$ simply the \emph{canonical connection}.
The connection $\widetilde{\nabla}$ induces the \emph{Hermitian curvature tensor}
$$
\widetilde{R}(X,Y,Z,W)=g(\wn_X\wn_Y Z-\wn_Y\wn_X Z-\wn_{[X,Y]}Z,W)\,.
$$
Since $\wn$ preserves $g$, one has
$$
\widetilde{R}(X,Y,Z,W)=-\widetilde{R}(Y,X,Z,W)=-\widetilde{R}(X,Y,W,Z)\,.
$$
Note that since $\wn$ has torsion, in general $\widetilde{R}$ does not satisfies the first
Bianchi identity \eqref{bianchi}. Moreover in general we don't have
$\widetilde{R}(X,Y,Z,W)=\widetilde{R}(Z,W,X,Y)$.

\subsection{The Gray conditions} In \cite{Gray} Gray considered some special classes of almost Hermitian manifolds characterized by some
identities involving the curvature tensor.
\begin{definition}\emph{
Let $(M,g,J)$ be an almost Hermtian manifold and let $R$ be
the curvature tensor of $g$. Then $R$ is said to satisfy}
\begin{itemize}
\item \emph{the \emph{first Gray identity} (G$_1$) if
$R(Z_1,Z_2,\cdot,\cdot)=0$};

\vspace{0.1cm} \item \emph{the \emph{second Gray identity} (G$_2$)
if $R(Z_1,Z_2,Z_3,Z_4)=R(\overline{Z}_1,Z_2,Z_3,Z_4)=0$};

\vspace{0.1cm} \item \emph{the \emph{third Gray identity} (G$_3$)
if $R(\overline{Z}_1,Z_2,Z_3,Z_4)=0$};
\end{itemize}
\emph{for every $Z_1,Z_2,Z_3,Z_4\in\Gamma(T^{1,0}M)$. }
\end{definition}
\noindent Clearly one has
$$
({\rm G}_1)\Longrightarrow({\rm G}_2)\Longrightarrow({\rm G}_3)
$$
and that the curvature tensor of a K\"ahler manifold satisfies
(G$_1$). Furthermore, in view of a Theorem of Goldberg (see
\cite{Goldberg}), any almost K\"ahler manifold whose curvature
tensor satisfies (G$_1$) is a genuine K\"ahler manifold. The same
can not be claimed for the condition $({\rm G}_2)$. Indeed in
dimension greater than $6$ there exist examples of compact
strictly almost K\"ahler manifolds whose curvature tensor
satisfies (G$_{2}$) (see \cite{davidov}). In dimension $4$ there is a different behavior
since we have the following theorem due to Apostolov,
Armstrong and  Dr\u aghici:
\begin{theorem}[\cite{AAD}, Theorem 2]\label{apostolov}
In dimension $4$ there are no compact strictly almost K\"ahler
manifolds whose curvature tensor satisfies $(\emph{G}_3)$.
\end{theorem}
\subsection{Generalized normal holomorphic frames}
Let $(M,g,J,\omega)$ be a $2n$-dimensional almost Hermitian manifold. Denote by $\nabla$ the Levi-Civita connection associated to the metric $g$, by
 $R$ the curvature tensors associated to $\nabla$ and by $N$ the Nijenhuis tensor of $J$.
\begin{definition}\emph{
Let $o$ be an arbitrary point in $M$. A \emph{generalized normal holomorphic frame} (or shortly a
\emph{g.n.h.f}) around $o$ is a local $(1,0)$-complex frame
$\{Z_1,\dots,Z_n\}$ satisfying the following properties:
\begin{enumerate}
\item[a.] $\nabla_{i}{Z}_{\ov{j}}(o)=0\,$;
\vspace{0.1 cm}
\item[b.] $\nabla_{i}Z_{j}(o)$ is of type $(0,1)$\,;
\vspace{0.1 cm}
\item[c.] $g_{i\ov{j}}(o)=\delta_{ij}$, ${\rm d} g_{i\ov{j}}(o)=0$\,;
\vspace{0.1 cm}
\item[d.] $\nabla_{i}\nabla_{\ov{j}}Z_{k}(o)=0$\,;
\end{enumerate}
for every $i,j,k=1,\dots,n$.
}
\end{definition}
\noindent We can recall the following
\begin{theorem}[\cite{LV1}, Theorem 1]\label{LV1}
The following facts are equivalent
\begin{enumerate}
\item [a.] $\omega$ is a quasi K\"ahler form;
\vspace{0.1 cm}
\item [b.] Any point $o$ in $M$ admits a generalized normal holomorphic frame.
\end{enumerate}
\end{theorem}
\noindent The following lemma, whose proof is similar to the one of Theorem 3.3 of \cite{LV2}, will be useful in the sequel
\begin{lemma}\label{nablaN}
Let $F$ be the smooth tensor on $M$ define by
$$
F(X,Y,Z,W):=g((\nabla_XN)(Y,Z),W)\quad\mbox{ for }X,Y,Z,W\in\Gamma(TM)\,.
$$
Consider an arbitrary point $o$ of $M$ and let $\{Z_1,\dots,Z_n\}$ be a g.n.h.f. around $o$. Then
$$
F_{\overline{i}jk\overline{l}}(o)=4g([Z_j,Z_k],\nabla_{\overline{i}}Z_{\overline{l}})(o)
$$
for every $i,j,k,l=1,\dots, n$.
\end{lemma}
\noindent The next result is a slight improvement of Theorem 3.3 of \cite{LV2} and can be viewed as a corollary of Lemma \ref{nablaN}
\begin{theorem}
Let $(M,g,J,\omega)$ be a quasi K\"ahler manifold and  assume that the Nijenhuis tensor of $J$ satisfies
\begin{equation}\label{SN}
\underset{X,Y,Z}{\mathfrak{S}}\,\nabla_X N(Y,Z)=0\,,\quad \forall\,X,Y,Z\in\Gamma(TM)\,,
\end{equation}
then $J$ is integrable.
\end{theorem}
\begin{proof}
Let $o\in M$ and let $\{Z_1,\dots,Z_n\}$ be a g.n.h.f. around $o$.  By \eqref{proprietadiN}, we have
$$
N_{i{\overline{k}}}(o)=0\,;\quad N_{ik}(o)\in T_o^{0,1}M\,,\;\;\mbox{ for every }\, i,j=1,\dots ,n\,.
$$
Furthermore, by the properties of the g.n.h.f., we have
$$
\underset{\ov{i},j,k}{\mathfrak{S}}\,(\nabla_{\overline{i}}N)(Z_j,Z_k)(o)= \nabla_{\overline{i}}(N(Z_j,Z_k))(o)\,.
$$
Hence equation \eqref{SN} implies $(\nabla_{\overline{i}} N)_{jk}=0$ which, in view of Lemma \ref{nablaN},
is equivalent to $N=0$.
\end{proof}
\noindent A direct computation gives the following
\begin{prop}
The components of the curvature tensor with respect to a g.n.h.f. $\{Z_1,\dots,Z_n\}$
around a point $o$ write as
$$
\begin{aligned}
&R_{i\overline{j}k\overline{l}}(o)=-g(\nabla_{\overline{j}}\nabla_{i}Z_{k},Z_{\overline{l}})(o)\,;\\
&R_{\overline{i}jkl}(o)=g(\nabla_{\overline{i}}\nabla_jZ_k,Z_l)(o)\,;\\
&R_{\overline{i}\overline{j}kl}(o)=-g(\nabla_{[Z_{\overline{i}},Z_{\overline{j}}]}Z_k,Z_l)(o)\,;\\
&R_{ijkl}(o)=g(\nabla_i\nabla_jZ_k,Z_l)(o)-g(\nabla_j\nabla_iZ_k,Z_l)(o)\,.
\end{aligned}
$$
\end{prop}
\subsection{Proof of formula \eqref{s-s*}}
The aim of this section is to give an alternative proof of formula \eqref{s-s*} without use the
Weitzenb\"ock decomposition:

\begin{proof}[Proof of formula  \eqref{s-s*}] Let  $(M,g,J,\omega)$ be an almost K\"ahler manifold. First of all we recall the definition of the $*$-Ricci tensor and the
$*$-scalar curvature
\noindent
$$
r_*(X,Y):=\sum_{i=1}^{2n}R(JX,JX_i,X_i,Y)\,,\quad s_*:=\sum_{i=1}^{2n}r_*(X_i,X_i)\,,
$$
where $\{X_1,\dots,X_{2n}\}$ is an arbitrary orthonormal frame on $M$.
It is easy to see that in complex coordinates the scalar curvature
and the $*$-scalar curvature write as
$$
\begin{aligned}
s=2\sum_{i,j=1}^n \{R_{i\overline{j}j\overline{i}}-R_{ij\overline{i}\overline{j}}\}\,,\quad
&s_*=2\sum_{i,j=1}^n \{R_{i\overline{j}j\overline{i}}+R_{ij\overline{i}\overline{j}}\}\,,
\end{aligned}
$$
being $\{Z_1,\dots,Z_n\}$ an arbitrary unitary $(1,0)$-frame on $M$.
In particular
$$
s_*-s=4\sum_{i,j=1}^n R_{ij\overline{i}\overline{j}}
$$
and formula \eqref{s-s*} can be rewritten as
$$
\sum_{i,j=1}^n R_{ij\overline{i}\overline{j}}=\frac14\|\nabla\omega\|^2\,.
$$
Fix an arbitrary point $o$ of $M$ and let $\{Z_1,\dots,Z_n\}$
be a g.n.h.f. around $o$. Since $\nabla_{i}Z_j(o)\in T^{0,1}_oM$,
then $N_{ij}(o)=-4[Z_i,Z_j](o)$; hence formula \eqref{fund2} reads
at $o$ as
$$
g([Z_i,Z_j],Z_l)(o)=-g(\nabla_{l}Z_i,Z_j)(o)\,.
$$
Since $\{Z_1,\dots,Z_n\}$ is an unitary frame we have
$$
[Z_i,Z_j](o)=-\sum_{l=1}^n\Gamma_{li}^{\ov{j}}(o)\,Z_{\overline{l}}(o)\,,
$$
where $\Gamma_{li}^{\ov{j}}:=g(\nabla_{l}Z_i,Z_j)$. Furthermore we have
$$
\begin{aligned}
R_{ij\overline{i}\overline{j}}(o)&=-g(\nabla_{[Z_i,Z_j]}Z_{\overline{i}},Z_{\overline{j}})(o)=\sum_{l=1}^n\Gamma_{li}^{\ov{j}}g(\nabla_{
\overline{l}}Z_{\overline{i}},Z_{\overline{j}})(o)\\
&=\sum_{l=1}^n\Gamma_{li}^{\ov{j}}(o)\Gamma_{\overline{l}\overline{i}}^{j}(o)
=\sum_{l=1}^n \vert \Gamma_{li}^{\ov{j}}\vert^2(o)\,.
\end{aligned}
$$
Hence
$$
\sum_{i,j=1}^nR_{ij\overline{i}\overline{j}}(o)=\sum_{l,i,j=1}^n\vert \Gamma_{li}^{\ov{j}}\vert^2(o)
$$
and the claim follows since
$
(\nabla_Z \omega)(X,Y)=\frac12g(N(X,Y),JZ)\,.
$
\end{proof}
Condition \eqref{s-s*} is related to the subspace $\mathcal{W}_4$ described in \cite[pag. 372]{Trix}
(see also \cite{Falcitelli} where $\mathcal{W}_4= {\mathcal{C}}_4$).
Indeed, by using Lemma 4.5 at page 371 in \cite{Trix} it is easy to see that the projection
$R^{\mathcal{W}_4}$ of $R$ to $\mathcal{W}_4$ is given by
\[ R^{\mathcal{W}_4} =\frac{(s - s_*)}{16n(n-1)} =\frac{1}{4n(n-1)}\sum_{i,j=1}^n R_{ij\overline{i}\overline{j}}=\frac{
1}{16n(n-1)}
\|\nabla\omega\|^2
 \,.
\]
\section{The first Bianchi identity for the Hermitian curvature}
\label{secR1}
\noindent In this section we are going to prove Theorem \ref{main} and its Corollary \ref{4Apost}.\\

Let $\widetilde{\nabla}$ be the canonical connection associated to a quasi K\"ahler structure
$(g,J,\omega)$ on a $2n$-dimensional manifold $M$. We have
\begin{lemma}\label{Z1Z2}
Let $Z_1,Z_2$ be two arbitrary $(1,0)$-vector fields on $M$. Then
$$
\wn_{Z_1}Z_2\in\Gamma(T^{1,0}M)\,,\quad \wn_{\overline{Z}_1}Z_2=\nabla_{\overline{Z}_1}Z_2\in\Gamma(T^{1,0}M)\,.
$$
\end{lemma}
\begin{proof} It is enough to consider the definition of $\wn$ and to apply Lemma \ref{LVcor}.
\end{proof}
\noindent As a direct consequence of Lemma \ref{Z1Z2} we have the following
\begin{prop}\label{RTILDE1}
Let $\{Z_1,\dots,Z_n\}$ be an arbitrary $(1,0)$-frame on $M$ and let $\widetilde{R}$ be the Hermitian curvature tensor of $M$. Then
\begin{enumerate}
\item[1.]$\widetilde{R}_{ijk\overline{l}}=R_{ijk\overline{l}}\,;$
\item[2.]$\widetilde{R}_{\overline{i}\overline{j}kl}=\widetilde{R}_{ijkl}=\widetilde{R}_{i\overline{j}kl}=0\,.$
\end{enumerate}
\end{prop}
\begin{lemma}\label{@o}
Let $o$ be an arbitrary point
of $M$ and let $\{Z_1,\dots,Z_n\}$ be a g.n.h.f. around $o$. Then
$$
\widetilde{\nabla}_{i}Z_j(o)=0\,,\quad  \widetilde{\nabla}_{\overline{i}}Z_j(o)=0\,,\quad \mbox{for any }i,j=1,\dots, n\,,
$$
i.e. the canonical connection acts on generalized normal holomorphic frames in quasi K\"ahler manifolds as the Levi-Civita connection
acts on normal holomorphic frames in K\"ahler manifolds.
\end{lemma}
\begin{proof}
Let $\{Z_1,\dots,Z_n\}$ be a g.n.h.f. around $o$. Since $\nabla_iZ_j(o)\in T^{0,1}_oM$, then we have
$$
\begin{aligned}
\wn_{i}Z_j(o)=&\frac12\{\nabla_i Z_j-J\nabla_i JZ_j\}(o)=\frac12\nabla_i Z_j(o)-\operatorname{i}\frac12J\nabla_i Z_j(o)\\
=&\frac12\nabla_i Z_j(o)-\frac12\nabla_i Z_j(o)=0\,.
\end{aligned}
$$
Moreover since $\nabla_{\overline{i}}Z_j(o)=0$, we have
\begin{equation*}
\wn_{\overline{i}}Z_j(o)=\frac12\{\nabla_{\overline{i}} Z_j-J\nabla_{\overline{i}} JZ_j\}(o)=\frac12\nabla_{\overline{i}} Z_j(o)
-\operatorname{i}\frac12J\nabla_{\overline{i}} Z_j(o)=0
\end{equation*}
and the claim follows.
\end{proof}
\noindent We have the following
\begin{prop}\label{RTILDE}
The components of the Hermitian curvature tensor
$\widetilde{R}$ with respect to a g.n.h.f. $\{Z_1,\dots,Z_n\}$
around a point $o$ write as
\begin{enumerate}
\item[1.]$\widetilde{R}_{i\overline{j}k\overline{l}}(o)=R_{i\overline{j}k\overline{l}}(o)-
g(\nabla_iZ_k,\nabla_{\overline{j}}Z_{\overline{l}})(o)\,;$
\item[2.]$\widetilde{R}_{ijk\overline{l}}(o)=R_{ijk\overline{l}}(o)\,;$
\item[3.]$\widetilde{R}_{\overline{i}\overline{j}kl}(o)=\widetilde{R}_{ijkl}(o)=\widetilde{R}_{i\overline{j}kl}(o)=0\,.$
\end{enumerate}
\end{prop}
\begin{proof}
The items $2$. and $3$. come from Proposition \ref{RTILDE1}.
The proof the first
identity can be obtained as follows: \\
By definition of $\widetilde{R}$ and the equation $[Z_i,Z_{\overline{j}}](o)=0$, we have
$$
\begin{aligned}
\widetilde{R}_{i\overline{j}k\overline{l}}(o)&=g(\wn_i\wn_{\overline{j}}Z_k-
\wn_{\overline{j}}\wn_iZ_k-\wn_{[Z_i,Z_{\overline{j}}]}Z_k,Z_{\overline{l}})(o)\\
&=g(\wn_i\wn_{\overline{j}}Z_k-\wn_{\overline{j}}\wn_iZ_k,Z_{\overline{l}})(o)\,.
\end{aligned}
$$
Applying Lemma \ref{Z1Z2} and Lemma \ref{@o}, we get
$$
\begin{aligned}
\widetilde{R}_{i\overline{j}k\overline{l}}(o)=&g(\wn_i\wn_{\overline{j}}Z_k-\wn_{\overline{j}}\wn_iZ_k,Z_{\overline{l}})(o)\\
=&g(\wn_i\nabla_{\overline{j}}Z_k,Z_{\overline{l}})(o)-g(\wn_{\overline{j}}\wn_iZ_k,Z_{\overline{l}})(o)\\
=&Z_ig(\nabla_{\overline{j}}Z_k,Z_{\overline{l}})(o)-g(\nabla_{\overline{j}}Z_k,\wn_iZ_{\overline{l}})(o)-Z_{\overline{j}}g(\wn_iZ_k,Z_{\overline{l}})(o)\\
&+g(\wn_iZ_k,\wn_{\overline{j}}Z_{\overline{l}})(o)\\
=&Z_ig(\nabla_{\overline{j}}Z_k,Z_{\overline{l}})(o)-Z_{\overline{j}}g(\wn_iZ_k,Z_{\overline{l}})(o)\,.
\end{aligned}
$$
Finally, taking into account Lemma \ref{LVcor} and that $\nabla$ and $\widetilde{\nabla}$ preserve $g$, we obtain
$$
\begin{aligned}
\widetilde{R}_{i\overline{j}k\overline{l}}(o)
=&Z_ig(\nabla_{\overline{j}}Z_k,Z_{\overline{l}})(o)-Z_{\overline{j}}g(\wn_iZ_k,Z_{\overline{l}})(o)\\
=&g(\nabla_i\nabla_{\overline{j}}Z_k,Z_{\overline{l}})(o)+g(\nabla_{\overline{j}}Z_k,\nabla_iZ_{\overline{l}})(o)\\
&-Z_{\overline{j}}Z_ig_{k\overline{l}}(o)+Z_{\overline{j}}g(Z_k,\wn_iZ_{\overline{l}})(o)\\
=&-Z_{\overline{j}}Z_ig_{k\overline{l}}(o)+Z_{\overline{j}}g(Z_k,\wn_iZ_{\overline{l}})(o)\\
=&-Z_{\overline{j}}g(\nabla_iZ_k,Z_{\overline{l}})(o)-Z_{\overline{j}}g(Z_k,\nabla_iZ_{\overline{l}})(o)\\
&-Z_{\overline{j}}g(Z_k,\nabla_iZ_{\overline{l}})(o)\\
=&-g(\nabla_{\overline{j}}\nabla_iZ_k,Z_{\overline{l}})(o)-g(\nabla_iZ_k,\nabla_{\overline{j}}Z_{\overline{l}})(o)\\
&-g(\nabla_{\overline{j}}Z_k,\nabla_iZ_{\overline{l}})(o)-g(Z_k,\nabla_{\overline{j}}\nabla_iZ_{\overline{l}})(o)\\
=&R_{i\overline{j}k\overline{l}}(o)-g(\nabla_iZ_k,\nabla_{\overline{j}}Z_{\overline{l}})(o)\,,
\end{aligned}
$$
i.e.
$$
\widetilde{R}_{i\overline{j}k\overline{l}}(o)=R_{i\overline{j}k\overline{l}}(o)-g(\nabla_iZ_k,\nabla_{\overline{j}}Z_{\overline{l}})(o)\,,
$$
and the claim follows.
\end{proof}

\noindent Now we are ready to prove Theorem $\ref{main}$:

\begin{proof}[Proof of Theorem $\ref{main}$]
Let $o\in M$ be an arbitrary point and let $\{Z_1,\dots,Z_n\}$ be a g.n.h.f. around $o$. By Proposition \ref{RTILDE}, we have
$$
\underset{i,j,k}{\mathfrak{S}}\widetilde{R}_{ijkl}(o)=\underset{i,j,k}{\mathfrak{S}}\widetilde{R}_{ijk\overline{l}}(o)=0\,.
$$
Moreover
\begin{equation}
\underset{i,\ov{j},k}{\mathfrak{S}}\widetilde{R}_{i\overline{j}kl}(o)=R_{ki\overline{j}l}(o)\,.
\end{equation}
Furthermore
$$
\begin{aligned}
\underset{i,\ov{j},k}{\mathfrak{S}}\widetilde{R}_{i\overline{j}k\overline{l}}(o)=&
\widetilde{R}_{i\overline{j}k\overline{l}}(o)+\widetilde{R}_{ki\overline{j}\overline{l}}(o)+\widetilde{R}_{\overline{j}ki\overline{l}}(o)\\
=&\widetilde{R}_{i\overline{j}k\overline{l}}(o)+\widetilde{R}_{\overline{j}ki\overline{l}}(o)\\
=&R_{i\overline{j}k\overline{l}}(o)+R_{\overline{j}ki\overline{l}}(o)-
g(\nabla_iZ_k,\nabla_{\overline{j}}Z_{\overline{l}})(o)+g(\nabla_kZ_i,\nabla_{\overline{j}}Z_{\overline{l}})(o)\\
=&-R_{ki\overline{j}\overline{l}}(o)-g([Z_i,Z_k],\nabla_{\overline{j}}Z_{\overline{l}})(o)\,,
\end{aligned}
$$
i.e.
\begin{equation}
\underset{i,\ov{j},k}{\mathfrak{S}}\widetilde{R}_{i\overline{j}k\overline{l}}(o)=R_{ik\overline{j}\overline{l}}(o)-g([Z_i,Z_k],\nabla_{\overline{j}}Z_{\overline{l}})(o)\,.
\end{equation}
Hence the Hermitian curvature $\widetilde{R}$ satisfies the first Bianchi identity at $o$ if and only if the following equations hold:
\begin{eqnarray}
&&\label{prima} R_{ki\overline{j}l}(o)=0\,;\\
&&\label{seconda}R_{ik\overline{j}\overline{l}}(o)-g([Z_i,Z_k],\nabla_{\overline{j}}Z_{\overline{l}})(o)=0\,.
\end{eqnarray}
Equation \eqref{prima} is  the third Gray condition, while, in view of Lemma
\ref{nablaN}, equation \eqref{seconda} is satisfied if and only if
$$
R(Z_1,Z_2,\overline{Z}_3,\overline{Z}_4)=\frac14g((\nabla_{\overline{Z}_3}N)(Z_1,Z_2),\overline{Z}_4)
$$
for every $Z_1,Z_2,Z_3,Z_4\in\Gamma(T^{1,0}M)$.
\end{proof}
\noindent Now we can prove Corollary $\ref{4Apost}$.
\begin{proof}[Proof of Corollary $\ref{4Apost}$] Assume that $(M,g,J,\omega)$ is an almost K\"ahler manifold and let $\widetilde{R}$ be the Hermitian
curvature of $(g,J)$. Fix an arbitrary point $o$ of $M$, consider a g.n.h.f. $\{Z_1,\dots,Z_n\}$
around $o$ and assume that $\widetilde{R}$ satisfies the
first Bianchi identity. Then, in view of Theorem \ref{main}, we have
$$
0=R_{ik\ov{j}\ov{l}}(o)-g([Z_i,Z_k],\nabla_{\ov{j}}Z_{\ov{l}})(o)=-g(\nabla_{[Z_i,Z_k]}Z_{\ov{j}},Z_{\ov{l}})(o)-g([Z_i,Z_k],\nabla_{\ov{j}}Z_{\ov{l}})(o)\,,
$$
i.e.
\begin{equation}\label{e}
g(\nabla_{[Z_i,Z_k]}Z_{\ov{j}},Z_{\ov{l}})(o)=-g([Z_i,Z_k],\nabla_{\ov{j}}Z_{\ov{l}})(o)\,.
\end{equation}
In particular
$$
g([Z_i,Z_k],\nabla_{\ov{j}}Z_{\ov{l}})(o)=-g([Z_i,Z_k],\nabla_{\ov{l}}Z_{\ov{j}})(o)\,,
$$
i.e. $g([Z_i,Z_k],\nabla_{\ov{j}}Z_{\ov{l}})(o)$ is skew-symmetric with respect to the indexes $\ov{j},\ov{l}$.
In view of formula \eqref{fund2}, we have
\begin{equation*}
\begin{aligned}
g(\nabla_{[Z_i,Z_k]}Z_{\ov{j}},Z_{\ov{l}})(o)=&\frac14g(N_{\ov{j}\ov{l}},[Z_i,Z_k])(o)=-g([Z_{\ov{j}},Z_{\ov{l}}],[Z_i,Z_k])(o)\\
                                             =&-2g([Z_i,Z_k],\nabla_{\ov{j}}Z_{\ov{l}})(o)
\end{aligned}
\end{equation*}
Hence equation \eqref{e} implies
$$
g([Z_i,Z_k],\nabla_{\ov{j}}Z_{\ov{l}})(o)=0
$$
which forces $J$ to be integrable.
\end{proof}

\section{The condition $\widetilde{R}=0$ in quasi K\"ahler manifolds}\label{secR2}
\noindent In this section we investigate the case $\widetilde{R}=0$. We start by considering the following preliminar
\begin{lemma}\label{preliminare}
Let $(M,g,J,\omega)$ be a quasi K\"ahler manifold. Then the following are equivalent:
\begin{enumerate}
\item[1.] the curvature tensor of the canonical connection associated to $(g,J)$ vanishes;
\vspace{0.1cm}
\item[2.]every $o\in M$ admits an open neighborhood $U$ and a complex unitary $(1,0)$-frame $\{Z_1,\dots,Z_n\}$ on $U$ such that
$$
\nabla_{i}Z_j\in\Gamma(T^{0,1}U)\,,\quad \nabla_{\overline{i}}Z_j=0\,,\quad i,j=1,\dots, n\,.
$$
\end{enumerate}
\end{lemma}
\begin{proof}
The condition $\widetilde{R}=0$ is equivalent to require that every point $o$ of $M$ admits an open neighborhood $U$ equipped with
a complex unitary $(1,0)$-frame $\{Z_1,\dots,Z_n\}$ such that
\begin{equation}\label{derivative}
\wn_{i}Z_j=0\,,\quad \wn_{\overline{i}}Z_j=0\,,\quad i,j=1,\dots, n\,.
\end{equation}
Since
$$
\wn_{i}Z_j=0=\frac12\nabla_iZ_j-\frac12J\nabla_{i}JZ_j=\frac12\nabla_iZ_j+\frac12\operatorname{i}J\nabla_iZ_j\,;
$$
and
$$
\wn_{\overline{i}}Z_j=0=\frac12\nabla_{\overline{i}}Z_j-\frac12J\nabla_{{\overline{i}}}JZ_j\frac12\nabla_{\overline{i}}Z_j+\frac12
\operatorname{i}J\nabla_{\overline{i}}Z_j\,,
$$
then \eqref{derivative} is equivalent to require that $\nabla_{i}Z_j\,,\nabla_{\overline{i}}Z_j\in\Gamma(T^{0,1}U)$ for every $i,j=1,\dots, n$. By
the assumption on $M$ to be quasi K\"ahler we have $\wn_{\overline{i}}Z_j=0$.
\end{proof}
\begin{rem}\label{remdeba}\emph{
Note that the second item of the previews Lemma in particular
 implies that if $g$ is an $\widetilde{R}$-flat quasi K\"ahler metric, then we can always find a local
unitary $(1,0)$-coframe $\{\zeta_1,\dots,\zeta_n\}$ such that
$$
\partial\zeta_i=\ovp\zeta_i=0\,,\quad i=1,\dots,n\,.
$$}
\end{rem}
We recall that
a $4$-dimensional quasi K\"ahler manifold is always almost K\"ahler. Hence, in view of Theorem \ref{main}, if a $4$-dimensional
quasi K\"ahler manifold has $\widetilde{R}=0$ then it is K\"ahler. In greater dimension things work differently:
\begin{theorem}
There exists a
quasi K\"ahler structure $(g_0,J_0,\omega_0)$ on the Iwasawa manifold with the following properties:
\begin{enumerate}
\item[1.] the Hermitian curvature of $(g_0,J_0)$ vanishes;
\vspace{0.1cm}
\item[2.] the Riemann curvature of $g_0$ satisfies the second Gray identity $({\rm G_{2})}$.
\end{enumerate}
\end{theorem}
\begin{proof}
Let $G$ be the \emph{complex Heisenberg} group
$$
G:=\left\{
\left(
\begin{array}{cccccc}
1  &z_1    &z_2  \\
0  &1      &z_3     \\
0  &0      &1\\
\end{array}
\right)\,:z_i\in\C\,,\operatorname{i}=1,2,3
\right\}
$$
and let $M$ be the compact manifold $M=G/\Gamma$, where $\Gamma$ is the co-compact lattice of $G$ formed by the matrices with integral entries.
Then $M$ is the \emph{Iwasawa manifold}.
It is well known that $M$ admits a global frame
$\mathcal{B}=\{X_1,X_2,X_3,X_4,X_5,X_6\}$ satisfying the following structure equations
$$
\begin{aligned}
&[X_1,X_2]=X_3\,,\quad [X_4,X_5]=-X_3
&[X_2,X_4]=X_6\,,\quad [X_5,X_1]=X_6\,.
\end{aligned}
$$
Let
$J_0$ be the almost complex structure defined on the basis $\mathcal{B}$ by
$$
\begin{aligned}
&J_0X_1=X_4\,,\quad &&J_0X_2=X_5\,,\quad &&J_0X_3=X_6\,,\\
&J_0X_4=-X_1\,,     &&J_0X_5=-X_2\,,     &&J_0X_6=-X_3\,,
\end{aligned}
$$
let $g_0$ be the $J_0$-almost Hermitian metric
$$
g_0=\sum_{i=1}^6\alpha_i\otimes\alpha_i\,,
$$
and let
$$
\omega_0:=\alpha_1\wedge\alpha_4+\alpha_2\wedge\alpha_5+\alpha_3\wedge\alpha_6\,,
$$
being $\{\alpha_1,\dots,\alpha_6\}$ the dual frame of $\mathcal{B}$. Then $(g_0,J_0,\omega_0)$ is a quasi K\"ahler structure on $M$.\\
The almost complex structure $J_0$ induces the $(1,0)$-frame
$$
Z_1=X_1-\operatorname{i}X_4\,,\quad Z_2=X_2-\operatorname{i}X_5\,,\quad Z_3=X_3-\operatorname{i}X_6\,.
$$
Clearly
$$
[Z_1,Z_2]=2\,Z_{\overline{3}}\,,\quad [Z_{\overline{1}},Z_{\overline{2}}]=2\,Z_{3}
$$
and all other brackets involving the vectors of the frame vanish.
Furthermore, a direct computation gives $\nabla_{\overline{i}}Z_j=0$, for $i,j=1,2,3$ and
\begin{equation*}
\begin{aligned}
&\nabla_{1}Z_1=0\,,               &&\nabla_{2}Z_1=-Z_{\overline{3}}\,, &&&&\nabla_{3}Z_1=Z_{\overline{2}}\,,\\
&\nabla_{1}Z_2=Z_{\overline{3}}\,,&&\nabla_{2}Z_2=0\,,                &&&&\nabla_{3}Z_2=Z_{\overline{1}}\,,\\
&\nabla_{1}Z_3=-Z_{\overline{2}}\,,&&\nabla_{2}Z_3=Z_{\overline{1}}\,,&&&&\nabla_{3}Z_3=0
\end{aligned}
\end{equation*}
where $\nabla$ is the Levi-Civita connection associated to $g_0$.
Hence $\nabla_{i}Z_j\in\Gamma(T^{0,1}M) $ and in view of Lemma \ref{preliminare}
the Hermitian curvature tensor of $(g_0,J_0)$ vanishes. Furthermore a straightforward application of our
formulae yields
that the curvature tensor associated to $g_0$ satisfies the second Gray identity.
\end{proof}
The Iwasawa manifold is (in some fashion) the unique example of a $6$-dimensional non-K\"ahler
almost complex nilmanifold admitting a quasi K\"ahler $\widetilde{R}$-flat metric. More precisely we have the following
\begin{theorem}
Let $(G,J)$ be a $6$-dimensional Lie group equipped with a left-invariant non-integrable
almost complex structure admitting a $J$-compatible quasi K\"ahler metric $g$ with vanishing Hermitian curvature tensor.
Then the Lie algebra of $G$ endowed with the almost complex structure induced by $J$
is isomorphic as complex Lie algebra to the one of the complex Heisenberg group
equipped with the almost complex structure induced by $J_0$.
\end{theorem}
\begin{proof}
Let $\mathfrak{g}$ be the Lie algebra of $G$. In view of Lemma \ref{preliminare} there exists a complex $(1,0)$-frame $\{Z_1,Z_2,Z_3\}$ on $\mathfrak{g}$
such that
$$
[Z_i,Z_j]=\sum_{k=1}^{3}A_{ij}^{\ov{k}}Z_{\ov{k}}\,,\quad [Z_i,Z_{\ov{j}}]=0\,,\quad i,j=1,2,3.
$$
Since $J$ is by hypothesis non-integrable, there exists at least a bracket different from zero. We may assume
$$
[Z_1,Z_2]\neq 0\,.
$$
Now we observe that $A_{12}^{\ov{3}}\neq 0$. Indeed, if by contradiction $A_{12}^{\ov{3}}=0$, then
$$
[Z_1,Z_2]=A_{12}^{\ov{1}}Z_{\ov{1}}+A_{12}^{\ov{2}}Z_{\ov{2}}
$$
and by the Jacobi identity
$$
\begin{aligned}
&0=[[Z_1,Z_2],Z_{\ov{1}}]=-A_{12}^{\ov{2}}[Z_{\ov{1}},Z_{\ov{2}}]\,,\\
&0=[[Z_1,Z_2],Z_{\ov{2}}]=-A_{12}^{\ov{1}}[Z_{\ov{1}},Z_{\ov{2}}]
\end{aligned}
$$
which implies $[Z_1,Z_2]=0$. Hence $A_{12}^{\ov{3}}$ has to be different from zero and, consequently,
$$
W_{1}:=Z_1\,,\quad W_{2}=Z_2\,\quad W_{3}:=\frac{1}{A_{\ov{1}\ov{2}}^{3}}(Z_{3}-A_{\ov{1}\ov{2}}^{1}Z_{1}-A_{\ov{1}\ov{2}}^{2}Z_{2})
$$
is a (1,0)-frame on $(\mathfrak{g},J)$. Such a frame satisfies
$$
[W_1,W_2]=W_{\ov{3}}\,.
$$
Finally, using again the Jacobi identity, we get
$$
\begin{aligned}
&0=[[W_1,W_{2}],W_{\ov{1}}]=-[W_{\ov{2}},W_{\ov{3}}]\,,\\
&0=[[W_1,W_{2}],W_{\ov{2}}]=-[W_{\ov{1}},W_{\ov{3}}]\,,
\end{aligned}
$$
i.e.
$$
[W_{2},W_{3}]=[W_1,W_{3}]=0
$$
which ends the proof.
\end{proof}
It is possible to find some non-equivalent quasi K\"ahler structures on the Iwasawa manifold having $\widetilde{R}=0$.
For instance we have the following example
\begin{ex}\label{nuovoesempio}\emph{It easy to show that the Iwasawa manifold $M$
admits a global coframe $\{\alpha_1,\dots,\alpha_6\}$  satisfying the following
structure equations
$$
\begin{aligned}
&\di \alpha_1=\di\alpha_3=-\alpha_1\wedge\alpha_2+ \alpha_4\wedge\alpha_5 -\alpha_2\wedge\alpha_3 +\alpha_5\wedge\alpha_6\,;\\
&\di\alpha_2=\di\alpha_5=0\,;\\
&\di\alpha_4=\di\alpha_6=-\alpha_2\wedge\alpha_4 + \alpha_1\wedge\alpha_5 -\alpha_3\wedge\alpha_5 +\alpha_2\wedge\alpha_6\,.
\end{aligned}
$$
Let $\{X_1,\dots,X_6\}$ be the frame dual to $\{\alpha_1,\dots,\alpha_6\}$ and consider the almost complex structure $J$
on $M$ defined on $\{X_1,\dots,X_6\}$ by
$$
\begin{aligned}
&JX_1=X_4\,,\quad &&JX_2=X_5\,,\quad &&JX_3=X_6\,,\\
&JX_4=-X_1\,,     &&JX_5=-X_2\,,     &&JX_6=-X_3\,.
\end{aligned}
$$
Let
$$
\omega:=\alpha_1\wedge \alpha_4+\alpha_2\wedge\alpha_5+\alpha_3\wedge \alpha_6\,;
$$
then a direct computation gives that $\omega$ is a $\ovp$-closed form compatible with $J$.
The basis $\{X_1,\dots X_6\}$ induces the complex $(1,0)$-frame
$$
Z_1=X_1-\operatorname{i}X_4\,,\quad Z_2=X_2-\operatorname{i}X_5\,,\quad Z_3=X_3-\operatorname{i}X_6\,.
$$
One easily gets
$$
[Z_1,Z_2]=2(Z_{\overline{1}}+Z_{\overline{3}})\,,\quad [Z_2,Z_3]=2(Z_{\overline{1}}+Z_{\overline{3}})\,,\quad [Z_1,Z_3]=0\,.
$$
Since $[Z_i,Z_{\overline{j}}]=0$ and $(g,J,\omega)$ is a quasi K\"ahler structure, in view of Lemma \ref{LVcor}
we have
$$
\nabla_{\overline{i}}Z_{j}=0\,,
$$
being $\nabla$ the Levi-Civita connection associated to the metric $g$. Furthermore a direct computation gives
\begin{equation*}
\begin{array}{llll}
&\nabla_{1}Z_1=-2Z_{\overline{2}}\,,&\nabla_{2}Z_1=-2Z_{\overline{3}}\,,  &\nabla_{3}Z_1=0\,,\\
&\nabla_{1}Z_2=2Z_{\overline{1}}\,, &\nabla_{2}Z_2=0\,,                   &\nabla_{3}Z_2=-2Z_{\ov{3}}\,,\\
&\nabla_{1}Z_3=2Z_{\ov{1}}\,,       &\nabla_{2}Z_{3}=2Z_{\overline{1}}\,, &\nabla_{3}Z_3=2Z_{\ov{2}}\,;\\
\end{array}
\end{equation*}
hence
$$
\nabla_{i}Z_j\in\Gamma(T^{0,1}M)\,,\quad \mbox{for every }i,j=1,2,3\,.
$$
By Lemma \eqref{preliminare}, we get that the Hermitian curvature tensor of $g$ vanishes. Also in this case a straightforward computation
gives that the curvature tensor of the metric $g$ satisfies the second Gray identity $({\rm G}_2)$.}
\end{ex}
\begin{rem}
\emph{
In the quasi K\"ahler case the condition $\widetilde{R}=0$ implies that the tensor $\mathcal{R}(g,J)$ described by \eqref{tosatti} vanishes.
Hence it is very natural to take into account the following problem:
\begin{itemize}
\item Does there exist a symplectic form $\omega'$ on the Iwasawa manifold taming the almost complex structure $J_0$ and such that
the pair $(\omega',J_0)$ induces an $\widetilde{R}$-flat quasi K\"ahler structure on $M$ ?
\end{itemize}
(This problem was suggested us by Valentino Tosatti).
The answer is unfortunately negative. In order to show this we fix a quasi K\"ahler $\widetilde{R}$-flat metric $g$ on the Iwasawa manifold
$M$ compatible with $J_0$. Then we can find a global unitary $(1,0)$-coframe $\{\zeta_1,\zeta_2,\zeta_3\}$ such that
\begin{equation}\label{forme}
\partial\zeta_i=\ovp\zeta_{i}=0\,,\quad  i=1,2,3\,.
\end{equation}
Assume that there exists a symplectic structure $\omega'$
taming $J_0$ and such that the pair $(\omega',J_0)$ induces the metric $g$. Then one necessary has
$$
\omega'=\omega+\beta+\ov{\beta}
$$
being $\omega$ the quasi K\"ahler form associated to $g$ and
$\beta$ a complex form of type $(3,0)$. Since ${\rm d}\omega'=0$ and ${\rm d}\omega$ is of type (3,0)+(0,3), then it has to be
$$
\ovp\beta=0\,,\quad {\rm d}(\beta+\ov{\beta})=f\,\zeta_1\wedge\zeta_2\wedge\zeta_3+g\,\zeta_{\ov{1}}\wedge\zeta_{\ov{2}}\wedge\zeta_{\ov{3}}\,,
$$
being $f,g$ smooth maps on $M$.
We can write $\beta=a\zeta_{12}+b\zeta_{23}+c\zeta_{13}$, where $a,b,c$ are smooth functions on $M$. Since $\ovp\beta=0$,
then equations \eqref{forme} imply that $a,b,c$ are $J_0$-holomorphic maps on $M$. Since $M$ is compact, $a,b,c$ have to be constant.
Finally using again \eqref{forme}, we obtain that such a $\beta$ cannot exist.
}
\end{rem}
As another application of Lemma \ref{preliminare} we have the following
\begin{prop}
Let $G$ be a Lie group equipped with a left-invariant quasi K\"ahler structure having the Hermitian
curvature tensor vanishing. Then $G$ is a $2$-step nilpotent Lie group.
\end{prop}
\begin{proof}
Let $\{g,J,\omega\}$ be a quasi K\"ahler structure on $G$.
In view of Lemma \ref{preliminare}, there exists a global unitary $(1,0)$-frame $\{Z_1,\dots,Z_n\}$ on $G$ such
that
$$
\nabla_{i}Z_{\ov{j}}=0\,,\quad \nabla_{i}Z_{j}\in\Gamma(T^{0,1}G)\,,\quad 1\leq i,j\leq n\,,
$$
where $\nabla$ is the Levi-civita connection associated to $g$.
In particular we have
$$
[Z_i,Z_{\ov{j}}]=0\,,\quad  [Z_i,Z_j]\in\Gamma(T^{0,1}G)\,,\quad 1\leq i,j\leq n
$$
and, consequently,
$$
[[Z_i,Z_{\ov{j}}],Z_k]=[[Z_i,Z_j],Z_k]=0\,,\quad \quad 1\leq i,j,k\leq n\,.
$$
The Jacobi identity in terms of $\{Z_1,\dots,Z_n\}$ reads as
$$
[[Z_i,Z_j],Z_{\ov{k}}]=0\,,\quad 1\leq i,j,k\leq n\,.
$$
Hence we get
$$
[[Z_i,Z_{\ov{j}}],Z_k]=[[Z_i,Z_j],Z_k]=[[Z_i,Z_j],Z_{\ov{k}}]=0\,,\quad \quad 1\leq i,j,k\leq n
$$
and that the Lie algebra of $G$ is $2$-step nilpotent.
\end{proof}
In view of Remark \ref{remdeba}, require that a quasi K\"ahler metric $g$
locally admits a complex unitary $(1,0)$-frame $\{\zeta_1,\dots,\zeta_n\}$ satisfying
$$
\partial\zeta_i=\ovp\zeta_i=0\,,\quad i=1,\dots,n
$$
is a bit less that require that the Hermitian curvature tensor of $g$ vanishes. Hence it is rather natural to wonder if an almost K\"ahler structure
can admit such a coframe. The answer is negative, since we have the following Proposition
\begin{prop}
Let $(M,g,J,\omega)$ be an almost K\"ahler manifold. Assume that $M$ admits a global unitary $(1,0)$-coframe $\{\zeta_1,\dots,\zeta_n\}$ satisfying
$$
\partial\zeta_i=\ovp\zeta_i=0\,,\quad i=1,\dots,n\,.
$$
Then $M$ is K\"ahler.
\end{prop}
\begin{proof}
Assume that such a coframe exists and let $\{Z_1,\dots,Z_n\}$ be the dual frame. Then we have
$$
[Z_i,Z_{\ov{j}}]=0\,,\quad [Z_i,Z_j]\in\Gamma(T^{0,1}M)\,,\quad i,j=1,\dots, n\,.
$$
In particular, we can write
$$
[Z_i,Z_j]=\sum_{k=1}^n A_{ij}^{\ov{k}}\,Z_{\ov{k}}
$$
and the Nijenhuis tensor of $J$ satisfies
$$
N(Z_i,Z_j)=-4\sum_{k=1}^n A_{ij}^{\ov{k}}Z_{\ov{k}}\,.
$$
Now we recall that the Nijenhuis tensor of an almost K\"ahler manifold always satisfies
$$
\underset{X,Y,Z}{\mathfrak{S}}\,g(N(X,Y),Z)=0\,.
$$
This formula in our case reads as
\begin{equation}\label{A}
A_{ij}^{\ov{k}}+A_{ki}^{\ov{j}}+A_{jk}^{\ov{i}}=0\,,\quad 1\leq i,j,k\leq n\,.
\end{equation}
Since the brackets of the form $[Z_i,Z_{\ov{j}}]$ vanish, then the Jacobi identity in terms of $Z_i$'s reads as
$$
[[Z_i,Z_j],Z_{\ov{r}}]=0\,,\quad 1\leq i,j,r\leq n\,,
$$
i.e.
$$
0=[[Z_i,Z_j],Z_{\ov{r}}]=
\sum_{k=1}^n[A_{ij}^kZ_{\ov{k}},Z_{\ov{r}}]=-\sum_{k=1}^n Z_{\ov{r}}(A_{ij}^{\ov{k}})\,Z_{\ov{k}}+\sum_{k,s=1}^n A_{ij}^{\ov{k}}\ov{A}_{kr}^{\ov{s}}Z_s\,.
$$
In particular one has
\begin{equation}\label{B}
\sum_{k=1}^n A_{ij}^{\ov{k}}\ov{A}_{kr}^{\ov{s}}=0\,,\quad 1\leq i,j,s,r\leq n\,.
\end{equation}
Using equations \eqref{A} and \eqref{B}, we get
$$
0=\sum_{k=1}^n A_{ij}^{\ov{k}}\ov{A}_{ki}^{\ov{j}}=-\sum_{k=1}^n \{A_{ij}^{\ov{k}}
\ov{A}_{ij}^{\ov{k}}-A_{ij}^{\ov{k}}\ov{A}_{kj}^{\ov{i}}\}=
-\sum_{k=1}^n |A_{ij}^{\ov{k}}|^2
$$
which forces $(M,g,J,\omega)$ to be a K\"ahler manifold.
\end{proof}

\end{document}